\documentclass[reqno]{amsart}
\usepackage{amssymb}
\usepackage{hyperref}

\numberwithin{equation}{section} 

plus 2pt minus 2pt

\allowdisplaybreaks[4]

\newtheorem{theorem}{Theorem}[section]
\newtheorem{lemma}[theorem]{Lemma}
\newtheorem{proposition}[theorem]{Proposition}
\newtheorem{corollary}[theorem]{Corollary}
\theoremstyle{definition}
\newtheorem{definition}[theorem]{Definition}
\newtheorem{example}[theorem]{Example}

\theoremstyle{remark}
\newtheorem{remark}[theorem]{Remark}

\numberwithin{equation}{section}

\begin{document}

\title{Diskcyclicity of sets of operators and applications}
\author{ Mohamed Amouch and Otmane Benchiheb}
\address{Mohamed Amouch and Otmane Benchiheb,
University Chouaib Doukkali.
Department of Mathematics, Faculty of science
Eljadida, Morocco}
\email{amouch.m@ucd.ac.ma}

\email{otmane.benchiheb@gmail.com}

\begin{abstract}
In this paper, we introduce and study the diskcyclicity and disk transitivity of a set of operators. We establish a diskcyclicity criterion and we give the relationship between this criterion and the diskcyclicity. As applications, we study the diskcyclicty of $C_0$-semigroups and $C$-regularized groups. We show that a diskcyclic $C_0$-semigroup exists on a complex topological vector space $X$ if and only if dim$(X)=1$ or dim$(X)=\infty$ and we prove that diskcyclicity and disk transitivity of a $C_0$-semigroups (resp $C$-regularized groups) are equivalent.
\end{abstract}


\maketitle

\section{Introduction and Preliminary}
Let $X$ be a complex topological vector space and $\mathcal{B}(X)$ the space of all continuous linear operators on $X$. By an operator, we always mean a continuous linear operator.

The most studied notion in linear dynamics is that of hypercyclicity$:$
An operator $T$ is called hypercyclic if there is some vector $x \in X$ such that the orbit of $x$ under $T$;
$$Orb(T,x)=\{T^nx \mbox{ : } n \in\mathbb{N}\},$$
is a dense subset of $X$, such a vector $x$ is called hypercyclic
for $T$. The set of all hypercyclic vectors for $T$ is denoted by $HC(T)$. The first example of hypercyclic operator was given by Rolewicz in \cite{Rolewicz}. 
He proved that if $B$ is a backward shift on the Banach space $\ell^p(\mathbb{N})$, then $\lambda B$ is a hypercyclic operator
for any complex number $\lambda$ satisfying $\vert\lambda\vert>1$.
 Another important notion in dynamical system is that of supercyclicity. This notion was introduced by Hilden and Wallen in \cite{HW}. An operator T is called supercyclic if there is a vector $x \in X$ such that the protective orbit of $x$ under $T$;
$$\mathbb{C}Orb(T,x)=\{\lambda T^nx \mbox{ : }\lambda\in\mathbb{C}\mbox{, } n \in\mathbb{N}\},$$
 is dense in $X$. In this case, the vector $x$ is said to be a supercyclic
vector for $T$. The set of all supercyclic vectors for $T$ is denoted by $SC(T)$. For the more detailed information on both hypercyclicity and supercyclicity,
the reader may refer to \cite{Bayart Matheron,Erdmann Peris}.

In the same spirit, since the operator $\lambda B$ is not hypercyclic whenever $\vert \lambda\vert\leq1$, we
are motivated to study the disk orbit. An operator $T$ is said to be diskcyclic if there is some vector $x\in X$ such that the disk orbit;
$$\mathbb{D}Orb(T,x)=\{\alpha T^{n}x\mbox{ : }\alpha\in\mathbb{D}\mbox{, } n\geq0\},$$
 is dense in $X$. Such vector $x$ is called a diskcyclic vector for $T$, and the set of all diskcyclic vectors for $T$ is denoted by $\mathbb{D}C(T)$.

An equivalent notion of the diskcyclicity in the case of a  a second countable complex topological vector space is that of disk transitivity.
An operator $T$ is said to be disk transitive if for each pair $(U,V)$ of nonempty open subsets of $X$ there exist $\alpha\in\mathbb{D}$ and $n\geq0$ such that 
$$\alpha T^{n}(U)\cap V\neq \emptyset.$$
The diskcyclic criterion; a sufficient
set of conditions for an operator to be diskcyclic, is one of most important characterization of the diskcyclicity. An operator $T$ is said to be satisfies the diskcyclicity criterion if there exist an increasing sequence of integers $(n_k)$, a sequence $(\alpha_{n_k})\subset\mathbb{D}\setminus\{0\}$, two dense sets $X_0$, $Y_0\subset X$ and a sequence of maps $S_{n_k}$ : $Y_0\longrightarrow X$ such that$:$
\begin{itemize}
\item[$(i)$] $\alpha_{n_k}T^{n_k}x\longrightarrow 0$ for any $x\in X_0$;
\item[$(ii)$] $\alpha_{n_k}^{-1}S_{n_k}y\longrightarrow 0$ for any $y\in Y_0$;
\item[$(iii)$] $T^{n_k} S_{n_k}y\longrightarrow y$ for any $y\in Y_0$.
\end{itemize}
For a general overview of diskcyclicity and related proprieties in linear dynamics see \cite{BKN,LZ,LZ1,WZ}.

Let $X$ and $Y$ be topological vector spaces. If $T\in\mathcal{B}(X)$ and $S\in\mathcal{B}(Y)$, then $T$ and $S$ are called quasi-conjugate or quasi-similar if there exists an operator $\phi$ : $X\longrightarrow Y$ with dense range such that $S\circ\phi=\phi\circ T.$ If $\phi$ can be chosen to be a homeomorphism, then $T$ and $S$ are called conjugate or similar, see \cite[Definition 1.5]{Erdmann Peris}.
A property $\mathcal{P}$ is said to be preserved under quasi-similarity if for all $T\in\mathcal{B}(X)$, if $T$ has property $\mathcal{P}$, then every operator $S\in\mathcal{B}(Y)$ that is quasi-similar to $T$ has also property $\mathcal{P}$, see \cite[Definition 1.7]{Erdmann Peris}. 


Recall that the strong operator topology (SOT for short) on $\mathcal{B}(X)$ is the topology with respect to which any $T\in \mathcal{B}(X)$ has a neighborhood basis consisting of sets of the form
$$\Omega=\{S\in\mathcal{B}(X) \mbox{ : }Se_i-Te_i\in U\mbox{, }i=1,2,\dots,k\}$$
 where $k\in\mathbb{N}$, $e_1,e_2,\dots e_k\in X$ are linearly independent and $U$ is a neighborhood of zero in $X$, see \cite{Conway}.
 
Recall from \cite{AOH} that a set $\Gamma$ of operators is called hypercyclic if there exists a vector $x$ in $X$ such that its orbit under $\Gamma$; 
$Orb(\Gamma,x)=\{Tx \mbox{ : }T\in\Gamma\}$, is dense in $X$. 
If the protective orbit of $x$ under $\Gamma$; $\mathbb{C}Orb(\Gamma,x)=\{\alpha Tx\mbox{ : }T\in\Gamma\mbox{, }\alpha\in\mathbb{C}\}$, is dense in $X$, then $\Gamma$ is said to be supercyclic, see \cite{AOS}. If 
span$\{Orb(\Gamma,x)\}$ is dense in $X$, then $\Gamma$ is said to be cyclic, see \cite{AOC}. In each case,
such a vector x is called a hypercyclic vector, a supercyclic vector, and a cyclic vector for $\Gamma$,
respectively.

In this paper, we introduce and study the notion of diskcyclicity for a set of operators which
generalize the notion of diskcyclicity for a single operator. We deal with diskcyclic sets, we prove that some properties known for one diskcyclic operator remain true for a diskcyclic set of operators.
It known that the set of diskcyclic vectors of a single operators is a $G_\delta$ type set. In section 2, we prove that this result holds for the set of diskcyclic vectors of a set of operators and we establish that diskcyclicity is preserved under quasi-similarity.

 In section 3, we introduce the notion of disk transitive sets, strictly disk transitive sets, diskcyclic transitive sets, and the notion of diskcyclic criterion for sets of operators. We give relations between this notions and the concept of diskcyclic sets of operators and we prove that this notions are preserved under quasi-similarity or similarity. 
 
 In section 4, we give applications for strongly continues semigroups of operators. We show that a diskcyclic strongly continues semigroup of operators exists on a complex topological vector space $X$ if and only if dim$(X)=1$ or dim$(X)=\infty$. Finally, we prove the equivalence between diskcyclicity and disk transitivity and we give necessary and sufficient conditions for a strongly continues semigroup of operators to be diskcyclic.
 
In section 5, we study the particular case when $\Gamma$ stands for a $C$-regularized group of operators. We give an example of a $C$-regularized group of operators on the complex field $\mathbb{C}$ and we prove that, if $(S(z))_{z\in\mathbb{C}}$ is a diskcyclic $C$-regularized group of operators and $C$ has dense range, then $(S(z))_{z\in\mathbb{C}}$ is disk transitive.

\section{ Diskcyclic sets of operators}
Our main definition in this paper is that of diskcyclic set of operators. This definition is a generalization of the definition of a single diskcyclic operator.
\begin{definition}
We say that a set $\Gamma\subset \mathcal{B}(X)$ is diskcyclic if there exists some $x\in X$ for which the disk orbit of $x$ under $\Gamma;$
$$\mathbb{D} Orb(\Gamma,x):=\{\alpha Tx \mbox{ : }\alpha\in\mathbb{D} \mbox{, }T\in\Gamma\},$$
 is a dense subset of $X$. 
Such vector $x$ is called a diskcyclic vector for $\Gamma$ or a diskcyclic vector. The set of all diskcyclic vectors for $\Gamma$ is denoted by $\mathbb{D}C(\Gamma)$.
\end{definition}
\begin{remark}
An operator $T\in\mathcal{B}(X)$ is diskcyclic if and only if the set
$\Gamma=\{T^n\mbox{ : }n\geq0\}$
is diskcyclic.
\end{remark}
A necessary bu not sufficient condition of diskcyclicity is due to the next proposition. 
\begin{proposition}\label{pr1}
If $x$ is a diskcyclic vector for a set $\Gamma$, then
$$\sup\{\Vert\alpha Tx \Vert \mbox{ : }\alpha\in\mathbb{D}\mbox{, }T\in\Gamma \}=+\infty. $$
\end{proposition}
\begin{proof}
Towards a contradiction, assume that
$ \sup\{\Vert\alpha Tx \Vert \mbox{ : }\alpha\in\mathbb{D}\mbox{, }T\in\Gamma \}=m<+\infty,$
and let $y\in X$ such that $\Vert y \Vert>m$. Since $x\in \mathbb{D}C(\Gamma)$, there exists sequences $\{k\}$ of $\mathbb{N},$ $\{\alpha_k\}$ of $\mathbb{D}$ and $\{T_k\}$ of $\Gamma$ such that
$\alpha_k T_kx\longrightarrow y.$
It follows that $\Vert y \Vert\leq m$ and this is a contradiction.
\end{proof}
\begin{remark}
If $\Gamma\subset\mathcal{B}(X)$ is bounded, then $\Gamma$ can not be diskcyclic. Indeed,
let $x\in X$. Since $\Gamma$ is bounded, $\sup\{\Vert\alpha Tx \Vert \mbox{ : }\alpha\in\mathbb{D}\mbox{, }T\in\Gamma \}<\infty.$
By Proposition \ref{pr1}, $\Gamma$ can not be diskcyclic.
\end{remark}
Let $X$ be a complex topological vector space and  $\Gamma$ a subset $ \mathcal{B}(X).$ We denote by $\{\Gamma\}^{'}$ the set of all elements of $\mathcal{B}(X)$ which commutes with every element of $\Gamma.$ That is
$$ \{\Gamma\}^{'}:=\{S\in\mathcal{B}(X) \mbox{ : }TS=ST \mbox{ for all }T\in\Gamma\}. $$
\begin{proposition}\label{p1}
Assume that $\Gamma\subset\mathcal{B}(X)$ is diskcyclic and let $T\in\mathcal{B}(X)$ be an operator with dense range. If $T\in \{\Gamma\}^{'}$, then $Tx\in \mathbb{D}C(\Gamma)$, for all $x\in \mathbb{D}C(\Gamma)$.
\end{proposition}
\begin{proof}
Let $O$ be a nonempty open subset of $X$. Since $T$ is continuous and of dense range, $T^{-1}(O)$ is a nonempty open subset of $X$. Let $x\in \mathbb{D}C(\Gamma)$, then there exist $\alpha\in\mathbb{D}$ and $S\in \Gamma$  such that $\alpha Sx\in T^{-1}(O)$, that is $\alpha T(Sx)\in O$. Since $T\in \{\Gamma\}^{'}$, it follows that
$$\alpha S(Tx)=\alpha T(S x)\in O.$$
Hence, $\mathbb{D}Orb(\Gamma,Tx)$  meets every nonempty open subset of $X$. From this, $\mathbb{D}Orb(\Gamma,Tx)$ is dense in $X$. That is, $Tx\in \mathbb{D}C(\Gamma)$. 
\end{proof}
\begin{remark}
If $x\in \mathbb{D}C(\Gamma)$, then $\alpha x\in\mathbb{D}C(\Gamma)$, for all $\alpha \in \mathbb{C}\setminus\{0\}.$
\end{remark}

Let $X$ and $Y$ be topological vector spaces and let $\Gamma\subset \mathcal{B}(X)$ and $\Gamma_1\subset \mathcal{B}(Y)$. Recall from \cite{AOC}, that $\Gamma$ and $\Gamma_1$ are called quasi-similar if there exists an operator $\phi$ : $X\longrightarrow Y$ with dense range such that for all $T\in\Gamma,$ there exists $S\in\Gamma_1$ satisfying $S\circ\phi=\phi\circ T$. If $\phi$ is a
homeomorphism, then $\Gamma$ and $\Gamma_1$ are called similar.

The following proposition shows that the diskcyclicity of sets of operators is preserved under quasi-similarity.
\begin{proposition}\label{14}
Assume that $\Gamma$ and $\Gamma_1$ are quasi-similar, then $\Gamma$ is diskcyclic in $X$ implies that $\Gamma_1$ is diskcyclic in $Y$. Moreover, 
$$\phi(\mathbb{D}C(\Gamma)\subset \mathbb{D}C(\Gamma_1).$$
\end{proposition}
\begin{proof}
Let $O$ be a nonempty open subset of $Y$, then $\phi^{-1}(O)$ is a nonempty open subset of $X$. If $x\in \mathbb{D}C(\Gamma)$, then there exist $\alpha\in\mathbb{D}$ and $T\in\Gamma$ such that $\alpha Tx\in \phi^{-1}(O)$, that is $\alpha\phi(Tx)\in O$. Let $S\in\Gamma_1$ such that $S\circ\phi=\phi\circ T$. Hence,
$$\alpha S(\phi x) =\alpha\phi(Tx)\in O.$$
Thus, $\mathbb{D}Orb(\Gamma_1,\phi x)$ meets every nonempty open subset of $Y$. From this, we deduce that $\mathbb{D}Orb(\Gamma_1,\phi x)$ is dense in $Y$. That is, $\Gamma_1$ is diskcyclic and $\phi x\in \mathbb{D}C(\Gamma_1)$.
\end{proof}
\begin{corollary}
Assume that $\Gamma$ and $\Gamma_1$ are similar, then $\Gamma$ is diskcyclic in $X$ if and only if $\Gamma_1$ is diskcyclic in $Y$. Moreover, $$\phi(\mathbb{D}C(\Gamma)= \mathbb{D}C(\Gamma_1).$$
\end{corollary}
\begin{proposition}\label{prop2}
Let $\Gamma\subset\mathcal{B}(X)$ and $(c_T)_{T\in \Gamma}\subset\mathbb{R}_{+}^{*}$. If $\{c_T T\mbox{ : }T\in \Gamma\}$ is diskcyclic, then for all $(k_T)_{T\in \Gamma}\subset\mathbb{R}_{+}^{*}$ such that $c_T\leq k_T$ for all $T\in\Gamma,$ $\{k_T T\mbox{ : }T\in \Gamma\}$ is diskcyclic.
\end{proposition}
\begin{proof}
Let $x$ be a diskcyclic vector for $\{c_T T\mbox{ : }T\in \Gamma\}$. Since $c_T\leq k_T$ for all $T\in\Gamma$, it follows that
$$ \mathbb{D} Orb(\{c_T T\mbox{ : }T\in \Gamma\},x)\subset \mathbb{D} Orb(\{k_T T\mbox{ : }T\in \Gamma\},x).$$
Indeed, let $\alpha c_T Tx\in\mathbb{D} Orb(\{c_T T\mbox{ : }T\in \Gamma\},x)$. Put $\beta=\alpha\frac{c_T}{k_T}$, then
$ \vert \beta \vert\leq1,$
that is $\beta\in\mathbb{D}$. Thus, 
$$\alpha c_T Tx=k_T \beta Tx\in\mathbb{D} Orb(\{k_T T\mbox{ : }T\in \Gamma\},x).$$
Since $\mathbb{D} Orb(\{c_T T\mbox{ : }T\in \Gamma\},x)$ is dense in $X$, it follows that $\mathbb{D} Orb(\{k_T T\mbox{ : }T\in \Gamma\},x)$ is dense in $X$. Hence $\{k_T T\mbox{ : }T\in \Gamma\}$ is diskcyclic in $X$ and $x$ is a diskcyclic vector of $\{k_T T\mbox{ : }T\in \Gamma\}$.
\end{proof}

Let $\{X_i\}_{i=1}^{n}$ be a family of complex topological vector spaces and let $\Gamma_i$ be a subset of $\mathcal{B}(X_i)$, for all $1\leq i\leq n$. Define 
$$\oplus_{i=1}^n X_i=X_1\times X_2\times\dots \times X_n=\{\oplus_{i=1}^n x_i=(x_1,x_2,\dots,x_n) \mbox{ : }x_i\in X_i\mbox{, }1\leq i\leq n\},$$
and
$$\oplus_{i=1}^n\Gamma_i=\{\oplus_{i=1}^n T_i=T_1\times T_2\times\dots \times T_n\mbox{ : }T_i\in\Gamma_i\mbox{, }1\leq i\leq n\},$$
where $T_1\times T_2\times\dots \times T_n$ is the operator defined in $\oplus_{i=1}^n X_i$ by
$$\begin{array}{ccccc}
\oplus_{i=1}^nT_i & : & \oplus_{i=1}^nX_i & \longrightarrow & \oplus_{i=1}^nX_i \\
 & & \oplus_{i=1}^nx_i & \longmapsto & \oplus_{i=1}^nT_i x_i. \\
\end{array}$$
\begin{proposition}\label{p4}
Let $\{X_i\}_{i=1}^{n}$ be a family of complex topological vector spaces and $\Gamma_i$ a subset of
$\mathcal{B}(X_i),$ for all $1\leq i\leq n$. If $\oplus_{i=1}^n\Gamma_i$ is a diskcyclic set in $\oplus_{i=1}^n X_i$, then $\Gamma_i$ is a diskcyclic set in $X_i$, for all $1\leq i\leq n$. Moreover, if $(x_1,x_2,\dots,x_n)\in \mathbb{D}C(\oplus_{i=1}^n\Gamma_i)$, then $x_i\in \mathbb{D}C(\Gamma_i)$, for all $1\leq i\leq n$.
\end{proposition}
\begin{proof}
Assume that $\oplus_{i=1}^n \Gamma$ is diskcyclic in $\oplus_{i=1}^n X_i$ and let $(x_1,x_2,\dots,x_n)$ be a diskcyclic vector for $\oplus_{i=1}^n\Gamma_i$. For all $1\leq i\leq n$, let $O_i$ be a nonempty open subset of $X_i$, then $ O_1\times O_2\times\dots\times O_n$ is a nonempty open subset of $\oplus_{i=1}^n X_i$. Since $\mathbb{D}Orb(\oplus_{i=1}^n\Gamma_i,\oplus_{i=1}^n x_i)$ is dense in $\oplus_{i=1}^n X_i$, there exist $\alpha\in\mathbb{D}$ and $T_i\in\Gamma_i$; $1\leq i\leq n$ such that
$$(\alpha T_1\times\alpha T_2\times\dots \times\alpha T_n)(x_1,x_2,\dots,x_n)=\alpha(T_1 x_1,T_2 x_2,\dots,T_n x_n) \in  O_1\times O_2\times\dots\times O_n,$$
that is $\alpha T_i x_i\in O_i$, for all $1\leq i\leq n$. Hence, $\Gamma_i$ is a diskcyclic set in $X_i$ and $x_i\in \mathbb{D}C(\Gamma_i)$, for all $1\leq i\leq n$. 
\end{proof}

The following proposition shows that the set of all diskcyclic operators of a set of operators is a $G_\delta$ type. Note that by $\mathbb{U}$ we mean the set  $\{y\in X \mbox{ : }\Vert y \Vert\geq 1\}.$
\begin{proposition}\label{p2}
Let $X$ be a second countable complex topological vector space and $\Gamma\subset \mathcal{B}(X).$
If $\Gamma$is diskcyclic, then
$$\mathbb{D}C(\Gamma)=\bigcap_{n\geq1}\left( \bigcup_{T\in \Gamma}\bigcup_{\beta\in\mathbb{U}} T^{-1}(\beta U_{n})\right),$$
where $(U_n)_{n\geq1}$ is a countable basis of the topology of $X$. As a consequence, $\mathbb{D}C(\Gamma)$ is a $G_\delta$ type set.
\end{proposition}
\begin{proof}
Let $x\in X$. Then, $x\in \mathbb{D}C(\Gamma)$ if and only if $\overline{\mathbb{D}Orb(\Gamma,x)}=X$.
 Equivalently, for all $n\geq 1,$ $U_n \cap \mathbb{D}Orb(\Gamma,x)\neq\emptyset,$ that is for all $n\geq 1$,
there exist $\lambda\in\mathbb{D}\setminus\{0\}$ and $T\in\Gamma$ such that $\lambda Tx\in U_n$. This is
 equivalent to the fact that for all $n\geq1,$ there exist $\beta\in\mathbb{U}$ and $T\in\Gamma$ such that
 $x\in T^{-1}(\beta U_n)$. Hence, $\displaystyle x\in \bigcap_{n\geq1}\left( \bigcup_{T\in \Gamma
}\bigcup_{\beta\in \mathbb{U}} T^{-1}(\beta U_{n})\right).$ 
 Since for all $n\geq 1,$ the set $\displaystyle \bigcup_{T\in \Gamma}\bigcup_{\beta\in \mathbb{U}} T^{-1}(\beta U_{n})$ is open, it follows that $\mathbb{D}C(\Gamma)$ is a $G_\delta$ type set.
\end{proof}
\section{Density and disk transitivity of sets of operators}
 In the following definition, we introduce the notion of disk transitivity of sets of operators which generalizes the notion of disk transitivity of a single operator.
\begin{definition}
A set $\Gamma\subset \mathcal{B}(X)$ is said to be disk transitive set if for any pair $(U,V)$ of nonempty open subsets of $X$, there exist $\alpha\in\mathbb{D}\setminus\{0\}$ and $T\in\Gamma$ such that 
$$T(\alpha U) \cap V\neq\emptyset.$$
\end{definition}
\begin{remark}
An operator $T\in\mathcal{B}(X)$ is disk transitive if and only if
$$\Gamma=\{T^n\mbox{ : }n\geq0\}$$
is a disk transitive.
\end{remark}

The following proposition proves that the disk transitivity of sets of operators is preserved under quasi-similarity.

\begin{proposition}\label{prop1}
Assume that $\Gamma\subset\mathcal{B}(X)$ and $\Gamma_1\subset \mathcal{B}(Y)$ are quasi-similar. If $\Gamma$ is disk transitive in $X$, then $\Gamma_1$ is disk transitive in $Y$.
\end{proposition}
\begin{proof}
Since $\Gamma$ and $\Gamma_1$ are quasi-similar, there exists a continuous map $\phi$ : $X\longrightarrow Y$ with dense range such that for all $T\in\Gamma,$ there exists $S\in\Gamma_1$ satisfying $S\circ\phi=\phi\circ T$. 
Let $U$ and $V$ be nonempty open subsets of $X$. Since $\phi$ is continuous and of dense range, $\phi^{-1}(U)$ and $\phi^{-1}(V)$ are nonempty and open. Since $\Gamma$ is disk transitive in $X$, there exist  $y\in \phi^{-1}(U)$ and $\alpha\in\mathbb{D},$ $T\in\Gamma$ with $\alpha Ty\in\phi^{-1}(V)$, which implies that $\phi(y)\in U$ and $\alpha \phi(Ty)\in V$. Let $S\in\Gamma_1$ such that $S\circ\phi=\phi\circ T$. Then, $\phi(y)\in U$ and $\alpha S\phi(y)\in V$. Thus, $\alpha S(U)\cap V\neq\emptyset.$ Hence, $\Gamma_1$ is disk transitive in $Y.$
\end{proof}
\begin{corollary}
Assume that $\Gamma\subset\mathcal{B}(X)$ and $\Gamma_1\subset\mathcal{B}(Y)$ are similar. Then, $\Gamma$ is disk transitive in $X$ if and only if $\Gamma_1$ is disk transitive in $Y$.
\end{corollary}

In the following result, we give necessary and sufficient conditions for a set of operators to be disk transitive.
\begin{theorem}\label{tt}
Let $X$ be a complex normed space and $\Gamma\subset\mathcal{B}(X)$. The following assertions are equivalent$:$ 
\begin{itemize}
\item[$(i)$] $\Gamma$ is disk transitive;
\item[$(ii)$] For each $x$, $y\in X,$ there exists sequences $\{k\}$ in $\mathbb{N}$, $\{x_k\}$ in $X$, $\{\alpha_k\}$ in $\mathbb{D}$ and $\{T_k\}$ in $\Gamma$ such that
$$x_k\longrightarrow x\hspace{0.3cm}\mbox{ and }\hspace{0.3cm}T_k(\alpha_k x_k)\longrightarrow y;$$
\item[$(iii)$] For each $x$, $y\in X$ and for $W$ a neighborhood of $0$, there exist $z\in X$, $\alpha\in\mathbb{D}$ and $T\in\Gamma$  such that 
$$x-z\in W \hspace{0.3cm}\mbox{ and }\hspace{0.3cm} T(\alpha z)-y\in W. $$ 
\end{itemize}
\end{theorem}
\begin{proof}$(i)\Rightarrow(ii)$
Let $x$, $y\in X$. For all $k\geq1$, let $U_k=B(x,\frac{1}{k})$ and $V_k=B(y,\frac{1}{k})$. Then $U_k$ and $V_k$ are nonempty open subsets of $X$. Since $\Gamma$ is disk transitive, there exist $\alpha_k\in\mathbb{D}$ and $T_k\in\Gamma$ such that $T_k(\alpha_k U_k)\cap V_k\neq\emptyset$. For all $k\geq1$, let $x_k\in U_k$ such that $ T_k(\alpha x_k)\in V_k$, then 
$$\Vert x_k-x \Vert<\frac{1}{k}\hspace{0.3cm}\mbox{ and }\hspace{0.3cm}\Vert T_k(\alpha x_k)-y \Vert<\frac{1}{k}$$
which implies that 
$$x_k\longrightarrow x\hspace{0.3cm}\mbox{ and }\hspace{0.3cm} T_k(\alpha x_k)\longrightarrow y.$$

$(ii)\Rightarrow(iii)$ Let $x$, $y\in X$. There exists sequences $\{k\}$ in $\mathbb{N}$, $\{x_k\}$ in $X$  $\{\alpha_k\}$ in $\mathbb{D}$ and $\{T_k\}$ in $\Gamma$ such that
$$x_k-x\longrightarrow 0\hspace{0.3cm}\mbox{ and }\hspace{0.3cm}T_k(\alpha_k x_k)-y\longrightarrow 0.$$
Let $W$ be a neighborhood of zero, then there exists $N\in\mathbb{N}$ such that 
$$x-x_k\in W\hspace{0.3cm}\mbox{ and }\hspace{0.3cm}T_k(\alpha_k x_k)-y\in W,$$ 
for all $k\geq N$.

$(iii)\Rightarrow(i)$ Let $U$ and $V$ be two nonempty open subsets of $X$. Then there exist $x$, $y\in X$ such that $x\in U$ and $y\in V$. Since for all $k\geq1$,  $W_k=B(0,\frac{1}{k})$ is a neighborhood of $0$,  there exist $z_k\in X,$ $\alpha_k\in\mathbb{D}$ and $T_k\in\Gamma$ such that 
$$\Vert x-z_k\Vert<\frac{1}{k}\hspace{0.3cm}\mbox{ and }\hspace{0.3cm}\Vert T_k(\alpha_k z_k)-y\Vert<\frac{1}{k}.$$
This implies that 
$$z_k\longrightarrow x\hspace{0.3cm}\mbox{ and }\hspace{0.3cm}T_k(\alpha_k z_k)\longrightarrow y.$$
 Since $U$ and $V$ are nonempty open subsets of $X$, $x\in U$ and $y\in V$, there exists $N\in\mathbb{N}$ such that $z_k\in U$ and $T_k(\alpha_k z_k)\in V$, for all $k\geq N.$ 
\end{proof}

Recall that an operator is diskcyclic if and only if it is disk transitive.
 Let $\Gamma$ be a subset of $\mathcal{B}(X)$. In whats follows, we prove that if $\Gamma$ is disk transitive set then $\Gamma$ is diskcyclic.
\begin{theorem}\label{t1}
Let $X$ be a second countable Baire complex topological vector space and $\Gamma\subset \mathcal{B}(X)$. The following assertions are equivalent$:$
\begin{itemize}
\item[$(i)$] $\mathbb{D}C(\Gamma)$ is dense in $X$;
\item[$(ii)$] $\Gamma$ is disk transitive.
\end{itemize}
In this case, $\Gamma$ is diskcyclic and the set of all diskcyclic vectors of $\Gamma$ is a dense $G_\delta$ type set.
\end{theorem}
\begin{proof} 
Since $X$ is a second countable Baire complex topological vector space, we can consider $(U_m)_{m\geq1}$ a countable basis of the topology of $X$.\\  
$(i)\Rightarrow (ii) :$ Assume that $\mathbb{D}C(\Gamma)$ is dense in $X$ and let $U$ and $V$ be two nonempty open subsets of $X$.  By Proposition \ref{p2}, we have 
$$\mathbb{D}C(\Gamma)=\bigcap_{n\geq1}\left(\bigcup_{\beta\in\mathbb{U}}\bigcup_{T\in \Gamma} T^{-1}(\beta U_{n})\right).$$
Hence, for all $n\geq 1,$ $\displaystyle A_n:=\bigcup_{\beta\in\mathbb{U}}\bigcup_{T\in \Gamma} T^{-1}(\beta U_{n})$ is dense in $X$. Thus, for all $n,$ $m\geq 1,$ we have $A_n\cap U_m\neq\emptyset$ which implies that for all $n,$ $m\geq 1,$ there exist $\beta\in\mathbb{D}\setminus\{0\}$ and $T\in \Gamma$ such that $T(\beta U_m)\cap U_n\neq\emptyset$. Since $(U_m)_{m\geq1}$ is a countable basis of the topology of $X$, it follows that $\Gamma$ is a disk transitive set.\\
$(ii)\Rightarrow (i) : $ Assume that $\Gamma$ is disk transitive. Let $n$, $m\geq 1$, then there exist
$T\in \Gamma$ and $\beta\in \mathbb{D}\setminus\{0\}$ such that $T(\beta U_m)\cap U_n\neq \emptyset$,
which implies that  $T^{-1}(\frac{1}{\beta} U_n)\cap U_m\neq \emptyset$. Hence, for all $n\geq1,$ we have
$\displaystyle\bigcup_{\beta\in\mathbb{U}}\bigcup_{T\in \Gamma} T^{-1}(\beta U_{n})$ is dense in $X$.
Since $X$ is a Baire space, it follows that 
$$\mathbb{D}C(\Gamma)=\bigcap_{n\geq1}\left(\bigcup_{\beta\in\mathbb{U}}\bigcup_{T\in \Gamma} T^{-1}(\beta U_{n})\right)$$
is a dense subset of $X$.
\end{proof}
The converse of Theorem \ref{t1} holds with some additional assumptions.
\begin{theorem}\label{so}
Assume that for all $T$, $S\in\Gamma$ with $T\neq S$, there exists $A\in\Gamma$ such that $T=AS$. Then $\Gamma$ is diskcyclic implies that $\Gamma$ is disk transitive.
\end{theorem}
\begin{proof}
Since $\Gamma$ is diskcyclic, there exists $x\in X$ such that $\mathbb{D}Orb(\Gamma,x)$ is a dense subset of $X$.
Let $U$ and $V$ be two nonempty open subsets of $X$, then there exist $\alpha$, $\beta\in\mathbb{D}$ with $\vert\alpha\vert\leq\vert\beta\vert$, and $T$, $S\in\Gamma$ such that
\begin{equation}\label{e3}
\alpha Tx\in U \hspace{0.6cm}\mbox{ and }\hspace{0.6cm}\beta Sx\in V. 
\end{equation}
There exists $A\in\Gamma$ such that $T=AS$. By (\ref{e3}), we have
$$\alpha A(Sx)\in U \hspace{0.2cm}\mbox{ and }\hspace{0.2cm}\beta A(Sx)\in  A(V)$$
which implies that $U\cap A(\frac{\beta}{\alpha} V)\neq \emptyset$. 
Hence, $\Gamma$ is disk transitive.
\end{proof}
In the following definition we introduce the notion of strictly disk transitivity of set of operators. The case of hypercyclicity (resp, supercyclicity) was introduced in \cite{AKH} (resp \cite{AOS}).
\begin{definition}
A set $\Gamma\subset \mathcal{B}(X)$  is called strictly disk transitive if for each pair of nonzero elements
$x,$ $y$ in $X$, there exist some $\alpha\in\mathbb{D}$ and some $T\in\Gamma$ such that $\alpha Tx=y.$
\end{definition}
\begin{remark}
 An operator $T\in \mathcal{B}(X)$ is strictly disk transitive if and only if
$\Gamma=\{T^n\mbox{ : }n\geq0\}$
is strictly disk transitive.
\end{remark}
\begin{proposition}
If $\Gamma$ is strictly disk transitive set, then it is disk transitive. As a consequence, if $\Gamma$ is strictly disk transitive set, then it is diskcyclic.
\end{proposition}
\begin{proof}
Assume that $\Gamma$ is a strictly disk transitive set. If $U$ and $V$ are two nonempty open subsets of $X$, then there exist $x,$  $y\in X$ such that
$x\in U$ and $y\in V$. Since $\Gamma$ is strictly disk transitive, there exist $\alpha\in\mathbb{D}$ and $T\in\Gamma$ such that $\alpha Tx=y.$ Hence,
$$ \alpha Tx\in \alpha T(U)\hspace{0.3cm} \mbox{ and }\hspace{0.3cm}\alpha Tx\in V. $$
Thus, $\alpha T(U)\cap V\neq\emptyset,$ which implies that $\Gamma$ is disk transitive. By Theorem \ref{t1}, we deduce that $\Gamma$ is diskcyclic.
\end{proof}

In the following proposition, we prove that strictly disk transitivity of sets of operators is preserved under similarity.
\begin{proposition}
Assume that $\Gamma\subset \mathcal{B}(X)$ and $\Gamma_1\subset \mathcal{B}(Y)$ are similar. Then $\Gamma$ is strictly disk transitive in $X$ if and only if $\Gamma_1$ is strictly disk transitive in $Y.$
\end{proposition}
\begin{proof}
Since $\Gamma$ and $\Gamma_1$ are similar, there exists a homeomorphism $\phi$ : $X\longrightarrow Y$ such that for all $T\in\Gamma,$ there exists $S\in\Gamma_1$ satisfying $S\circ\phi=\phi\circ T$. Assume that $\Gamma$ is strictly disk transitive in $X$.
Let $x$, $y\in Y$. There exist $a,$ $b\in X$ such that $\phi(a)=x$ and $\phi(b)=y$. Since $\Gamma$ is
strictly disk transitive in $X$, there exists $\alpha\in\mathbb{D}$ and $T\in \Gamma$ such that $\alpha
Ta=b$, this implies that $\alpha(\phi\circ T)(a)=\phi(b)$. Since $\Gamma$ and $\Gamma_1$ are similar, it follows that there exists $S\in\Gamma_1$ such that $S\circ\phi=\phi\circ T$. Hence, $\alpha Sx=y$. Thus, $\Gamma_1$
is strictly disk transitive in $Y$.

For the converse, we do the same proof using $\phi^{-1}$ the invertible operator of $\phi$ and the proof is completed.
\end{proof}
 
Let $x$ be an element of a complex topological vector space $X$. Define
$\mathbb{D}_x:=\{\alpha x\mbox{ : }\alpha\in\mathbb{D}\}.$
\begin{theorem}
For each pair of nonzero vectors $x$, $y\in X$ with $y\notin \mathbb{D}_x$, there exists
a SOT-dense set $\Gamma_{xy}\subset\mathcal{B}(X)$ that is not strictly disk transitive. Furthermore, $\Gamma\subset\mathcal{B}(X)$ is a dense nonstrictly
disk transitive set if and only if $\Gamma$ is a dense subset of $\Gamma_{xy}$ for some $x$, $y\in X.$
\end{theorem}
\begin{proof}
Fix nonzero vectors $x,$ $y\in X$ such that $y\notin\mathbb{D}_x$ and put 
$\Gamma_{xy}=\{T\in\mathcal{B}(X) \mbox{ : }y\notin\mathbb{D}_{Tx}\}.$

It is clear that $\Gamma_{xy}$ is not strictly disk transitive. Let $\Omega$ be a nonempty
SOT-open set in $\mathcal{B}(X)$ and $S\in\Omega$. If $Sx$ and $y$ are such that
$y\notin\mathbb{D}_{Sx}$, then $S\in\Omega\cap\Gamma_{xy}$.
Otherwise, putting $S_n = S+\frac{1}{n}I$, we see that $S_k\in\Omega$ for some $k$, but $S_kx$ and $y$ are such that
$y\notin\mathbb{D}_{S_kx}$. Hence, $\Omega\cap\Gamma_{xy}\neq\emptyset$ and the proof is completed.

We prove the second assertion of the theorem. Suppose that $\Gamma$ is a dense subset of $\mathcal{B}(X)$ that is not
strictly disk transitive. Then there are nonzero vectors $x,$ $y\in X$ such that $y\notin\mathbb{D}_{Tx}$
for all $T\in \Gamma$ and hence $\Gamma\subset \Gamma_{xy}$.
To show that $\Gamma$ is dense in $\Gamma_{xy}$, assume that $\Omega_0$ is an open subset of 
$\Gamma_{xy}$. Thus, $\Omega_0= \Gamma_{xy}\cap\Omega$ for some open
set $\Omega$ in $\mathcal{B}(X)$. Then $\Gamma\cap \Omega_0= \Gamma\cap \Omega\neq\emptyset$.

For the converse, let $\Gamma$ be a dense subset of $\Gamma_{xy}$ for some $x,$ $y\in X$. Then $\Gamma$ is not strictly disk transitive. Also, since $\Gamma_{xy}$ is a dense subset of $\mathcal{B}(X)$, we conclude that $\Gamma$ is also dense in $\mathcal{B}(X)$. Indeed,
if $\Omega$ is any open set in $\mathcal{B}(X)$ then $\Omega\cap\Gamma_{xy}\neq\emptyset$ since $\Gamma_{xy}$ is dense in $\mathcal{B}(X)$. On the other hand, $\Omega\cap\Gamma_{xy}$
is open in $\Gamma_{xy}$ and so it must intersect $\Gamma$ since $\Gamma$ is dense in $\Gamma_{xy}$. Thus,  
$\Omega\cap\Gamma\neq\emptyset$ and so $\Gamma$ is dense in $\mathcal{B}(X)$.
\end{proof}

In the following definition we introduce the notion of diskcyclic transitivity of set of operators. The case of hypercyclicity (resp, supercyclicity) was introduced in \cite{AKH} (resp \cite{AOS}).
\begin{definition}
A set $\Gamma\subset \mathcal{B}(X)$ is said to be diskcyclic transitive if 
$$\mathbb{D}C(\Gamma)=X\setminus\{0\}.$$
\end{definition}
\begin{remark}
Let $X$ be a complex topological vector space. An operator $T\in \mathcal{B}(X)$ is diskcyclic transitive as an operator if and only if the set
$$\Gamma=\{T^n\mbox{ : }n\geq0\} $$
is diskcyclic transitive as a set of operators.
\end{remark}

It is clear that a diskcyclic transitive set is diskcyclic. Moreover, the next proposition shows that diskcyclic transitivity of sets of operators implies disk transitivity.
\begin{proposition}
If $\Gamma$ is diskcyclic transitive, then it is disk transitive. 
\end{proposition}
\begin{proof}
Let $U$ and $V$ be two nonempty open subsets of $X$. There exists $x\in X\setminus\{0\}$ such that $x\in U$. Since $\Gamma$ is diskcyclic transitive, there exists $\alpha\in\mathbb{D}$ and $T\in\Gamma$ such that $\alpha Tx\in V$. This implies that $\alpha T(U)\cap V\neq\emptyset.$ Hence, $\Gamma$ is disk transitive.
\end{proof}
\begin{remark}
Let $X$ be a complex topological vector space and $\Gamma$ a subset of $\mathcal{B}(X)$. Assume that $X$ is without isolated point and $\Gamma$ is diskcyclic transitive. To prove that $\Gamma$ is disk transitive 
one can remarks that $\overline{X \setminus\{0\}}=X$ and use Theorem \ref{t1}.
\end{remark}

The diskcyclic transitivity is preserved under similarity as show the next proposition.
\begin{proposition}
Assume that $\Gamma\subset \mathcal{B}(X)$ and $\Gamma_1\subset \mathcal{B}(Y)$ are similar. Then, $\Gamma$ is a diskcyclic transitive set in $X$ if and only if $\Gamma_1$ is a diskcyclic transitive set in $Y$.
\end{proposition}
\begin{proof}
Since $\Gamma$ and $\Gamma_1$ are similar, there exists a homeomorphism $\phi$ : $X\longrightarrow Y$ such that for all $T\in\Gamma,$ there exists $S\in\Gamma_1$ satisfying $S\circ\phi=\phi\circ T$. 
If $\Gamma$ is a diskcyclic transitive in $X$, then by Proposition \ref{14}, 
$$\phi(\mathbb{D}C(\Gamma))\subset \mathbb{D}C(\Gamma_1).$$
 Since $\phi$ is homeomorphism, the result holds.

For the converse, we do the same proof using $\phi^{-1}$ the invertible operator of $\phi$, and the proof is completed.
\end{proof}

Assume that $X$ is a topological vector space and $\Gamma$ a subset of $\mathcal{B}(X)$. The following result shows that the SOT-closure of $\Gamma$ is not large enough than $\Gamma$ to have more diskcyclic vectors.
\begin{proposition}\label{prop3}
If $\overline{\Gamma}$ stands for the SOT-closure of $\Gamma$ then $\Gamma$ is diskcyclic if and only if $\overline{\Gamma}$ is diskcyclic. Moreover, $\Gamma$ and $\overline{\Gamma}$ have the same diskcyclic vectors, that is
 $$\mathbb{D}C(\Gamma)=\mathbb{D}C(\overline{\Gamma}).$$
\end{proposition}
\begin{proof}
We only need to prove that 
$\mathbb{D}C(\overline{\Gamma})\subset \mathbb{D}C(\Gamma).$
 Fix $x\in \mathbb{D}C(\overline{\Gamma})$ and let $U$ be an arbitrary open subset of $X$. Then there is some $\alpha\in\mathbb{D}$ and $T\in\overline{\Gamma}$ such that $\alpha Tx \in U$. The set 
 $\Omega = \{S\in\mathcal{B}(X) \mbox{ : } \alpha Sx \in U\}$
 is a SOT-neighborhood
of $T$ and so it must intersect $\Gamma$. Therefore, there is some $S\in \Gamma$ such that $\alpha Sx \in U$ and this shows that
$x\in \mathbb{D}C(\Gamma)$.
\end{proof}
\begin{corollary}
A set $\Gamma$ is diskcyclic transitive if and only if $\overline{\Gamma}$ is diskcyclic transitive.
\end{corollary}
\begin{proof}
Assume that $\overline{\Gamma}$ is diskcyclic transitive, then 
$\mathbb{D}C(\overline{\Gamma})=X\setminus\{0\}$.
Since by Proposition \ref{prop3}, $\mathbb{D}C(\Gamma)=\mathbb{D}C(\overline{\Gamma})$, it follows that $\mathbb{D}C(\Gamma)=X\setminus\{0\}$. Hence, $\Gamma$ is diskcyclic transitive. The implication $\Gamma$ is diskcyclic implies $\overline{\Gamma}$ is diskcyclic is obvious which complete the proof.
\end{proof}
\begin{corollary}
A set $\Gamma$ is disk transitive if and only if $\overline{\Gamma}$ is disk transitive.
\end{corollary}
\begin{proof}
Direct consequence of Theorem \ref{t1} and Proposition \ref{prop3}.
\end{proof}

In the next definition, we introduce the diskcyclic criterion of sets of operators which generalizes the definition of diskcyclic criterion of operators.
\begin{definition}\label{cc}
We say that $\Gamma\subset\mathcal{B}(X)$ satisfies the criterion of diskcyclicity if there exist two dense subsets $X_0$ and $Y_0$ of $X$ and sequences $\{k\}$ of positives integers, $\{\alpha_k\}$ of $\mathbb{D}\setminus\{0\}$, $\{T_k\}$ of $\Gamma$ and a sequence of maps $S_k$ : $Y_0\longrightarrow X$  such that$:$
\begin{itemize}
\item[$(i)$] $\alpha_k T_kx\longrightarrow 0$ for all $x\in X_0$;
\item[$(ii)$] $\alpha_k^{-1} S_kx\longrightarrow 0$ for all $y\in Y_0$;
\item[$(iii)$] $T_kS_ky\longrightarrow y$ for all $y\in Y_0$.
\end{itemize}
\end{definition}
\begin{remark}
An operator $T\in\mathcal{B}(X)$ satisfies the criterion of diskcyclicity if and only if
$\Gamma=\{T^n \mbox{ : }n\geq0\} $
satisfies the criterion of diskcyclicity.
\end{remark}
\begin{theorem}\label{11}
Let $X$ be a second countable Baire complex topological vector space and $\Gamma$ a subset of $\mathcal{B}(X)$. If $\Gamma$ satisfies the criterion of diskcyclicity, then $\mathbb{D}C(\Gamma)$ is a dense subset of $X$. As consequence; $\Gamma$ is diskcyclic.
\end{theorem}
\begin{proof}
Assume that $\Gamma$ satisfies the diskcyclicity criterion. Let $U$ and $V$ be two nonempty open subsets of $X$. Since $X_0$ and $Y_0$ are dense in $X$, there exist $x_0$ and $y_0$ in $X$ such that
$$ x_0\in X_0\cap U\hspace{0.3cm}\mbox{ and }\hspace{0.3cm} y_0\in Y_0\cap V.$$
For all $k\geq1,$ let $z_k=x_0+\alpha_k^{-1} S_ky $. We have 
$\alpha_k^{-1} S_k y\longrightarrow 0$, which implies that $z_k\longrightarrow x_0$. Since $x_0\in U$
and $U$ is open, there exists $N_1\in\mathbb{N}$ such that $z_k\in U$, for all $k\geq N_1$. On the
other hand, we have $\alpha_k T_k z_k=\alpha_k T_k x_0+T_k (S_k y_0)\longrightarrow y_0$. Since $y_0\in V$
and $V$ is open, there exists $N_2\in\mathbb{N}$ such that $\alpha_k T_k z_k\in V$,
for all $k\geq N_2$. Let $N=$max$\{N_1,N_2\}$, then  $z_k\in U$ and $\alpha_k T_k z_k\in V$,
for all $k\geq N$, that is 
$$\alpha_k T_k (U)\cap V\neq \emptyset,$$
for all $k\geq N$.
Hence, $\Gamma$ is disk transitive. By Theorem \ref{t1} we deduce that $\mathbb{D}C(\Gamma)$ is a dense subset of $X$. We use again Theorem \ref{t1} to conclude that $\Gamma$ is a diskcyclic set. 
\end{proof}

\section{Diskcyclic strongly continuous semigroups of operators}
In this section we will study the case when $\Gamma$ stands for a strongly continuous semigroup of operators.

Recall that a one-parameter family $(T_t)_{t\geq0}$ of operators on a complex topological vector space $X$ is called a strongly continuous semigroup of operators if the following three conditions are satisfied$:$
\begin{itemize}
\item[$(i)$] $T_0=I$ the identity operator on $X$;
\item[$(ii)$] $T_{t+s}=T_{t}T_{s}$ for all $t,$ $s\geq0$;
\item[$(iii)$] $\lim_{t\rightarrow s}T_{t}x=T_{s}x$ for all $x\in X$ and $t\geq 0$.
\end{itemize}
One also refers to it as a $C_0$-semigroup.

The linear operator defined in 
$$ D(A)=  \left\lbrace x\in X\mbox{ : }\lim_{t\downarrow0}\frac{T_tx-x}{t}\mbox{ exists }\right\rbrace   $$
by
$$ Ax=\lim_{t\downarrow0}\frac{T_tx-x}{t}=\frac{d^{+}T_tx}{dt}\vert_{t=0}\mbox{, for }x\in D(A) $$
is the infinitesimal generator of the strongly continuous semigroup $(T_t)_{t\geq0}$ and $D(A)$ is the domain of $A$.
For more informations about the theory of strongly continuous semigroups the reader may refer to \cite{Pazy}.

The next example shows that there is a diskcyclic strongly continuous semigroups of operators on the complex field.
\begin{example}\label{ex}
Let $X=\mathbb{C}$. For all $t\geq0$, let $T_t$ be an operator defined by 
$$\begin{array}{ccccc}
T_t & : & \mathbb{C} & \longrightarrow & \mathbb{C} \\
 & & x & \longmapsto & \exp(t)x. \\
\end{array}$$
Then $(T_t)_{t\geq0}$ is strongly continuous semigroups of operators and we have
$$ \mathbb{D}Orb((T_t)_{t\geq0},1)=\{\alpha T_t(1)\mbox{ : }t\geq0\mbox{, }\alpha\in\mathbb{D}\}=\{\alpha y \mbox{ : }y\in\left[1,+\infty\right[  \mbox{, }\alpha\in\mathbb{D}\}.$$
Hence,
$$\overline{\mathbb{D}Orb((T_t)_{t\geq0},1)}=\mathbb{C}.$$
Thus, $(T_t)_{t\geq0}$ is a diskcyclic strongly continuous semigroups of operators and $1$ is a diskcyclic   vector of $((T_t)_{t\geq0})$.
\end{example}
\begin{remark}
Since all complex topological vector spaces of dimension one are isomorphe, we can deduce, by Using Example \ref{ex}, that there exists a diskcyclic strongly continuous semigroups of operators on each one dimensional space. 
\end{remark}

Recall from \cite[Lemma 5.1]{Wengenroth}, that if $X$ is a complex topological vector space such that $2 \leq$ dim$(X) <\infty$. Then $X$ supports no supercyclic strongly continuous semigroups of operators.

In the following theorem, we prove that the same results holds in the case of diskcyclicity.
\begin{theorem}
Assume that that $2 \leq$ dim$(X) <\infty$. Then $X$ supports no diskcyclic strongly continuous semigroups.
\end{theorem}
\begin{proof}
We use \cite[Lemma 5.1]{Wengenroth} and the fact that 
$\mathbb{D}Orb(T,x)\subset \mathbb{C}Orb(T,x)$ for all $x\in X$.
\end{proof}

A necessary and sufficient condition for a strongly continuous semigroups of operators to be diskcyclic is  due to next lemma and theorem. For the hypercyclicity (resp, the supercyclicity) version see \cite[Theorem 2.2.]{DSW} (resp, \cite[Lemma 1]{MT} and \cite[Lemma 2]{MT}).
\begin{lemma}\label{lm1}
Let $(T_t )_{t\geq0}$ be a diskcyclic strongly continuous semigroups of operators on a complex Banach infinite dimensional space $X$. If $x \in X$ is a diskcyclic vector of $(T_t )_{t\geq0}$, then the following assertions hold:
\begin{itemize}
\item[$(1)$] $T_t x\neq 0$ for all $t\geq 0$;
\item[$(2)$] The set $\{\alpha T_t x \mbox{ : }t\geq s\mbox{, }\alpha\in\mathbb{D}\}$ is dense in $X$ for all $s\geq 0$.
\end{itemize}
\end{lemma}
\begin{proof}
$(1)$ Suppose that $t_0>0$ is minimal with the property that $T_{t_0}x=0$.
First we show that for all $y \in X$,
there exists some $t\in[0,t_0]$ and $\alpha\in\mathbb{D}$ such that $y = \alpha T_tx$.
Since $x$ is a diskcyclic $\Gamma$, there exist a sequence $(t_n)_{n\in\mathbb{N}}$ of $[0,t_0]$ and a sequence $(\alpha_n)_{n\in\mathbb{N}}$ of $\mathbb{D}$ such that 
$$\alpha_n T_{t_n}x\longrightarrow y.$$
Since $[0,t_0]$ is compact we may suppose that $(t_n)_{n\in\mathbb{N}}$ converges to some $t$.
Without loss of generality we may assume that $(\alpha_n)_{n\in\mathbb{N}}$ converges to some $\alpha$ and we infer that $y=\alpha T_t x$.

Now let $y_i=\alpha_i T_{t_i}x\in X$, spanning a two-dimensional subspace,
such that each pair $y_i$, $y_j$, $i \neq j$, is linearly independent. Assume that $t_1 > t_2 > t_3$.
We have then $y_3 = c_1 y_1 + c_2 y_2$. We have
\begin{align*}
0&\neq \alpha_3 T_{(t_0+t_3-t_2)}x=T_{(t_0-t_2)}y_3=c_1T_{(t_0-t_2)}y_1+c_2T_{(t_0-t_2)}y_2\\
 &=c_1 \alpha_1 T_{(t_0+t_1-t_2)}x+c_2\alpha_2 T_{t_0}x=0.
\end{align*}
which is a contradiction.\\
$(2)$ Suppose that there exists some $s_0>0$ such that the set
$$A:=\{\alpha T_t x \mbox{ : }t\geq s_0\mbox{, }\alpha\in\mathbb{D}\}$$ 
is not dense in X. Hence there exists a bounded open set $U$ such that $U \cap \overline{A}= \emptyset$.
Therefore we have
$$ U\subset \overline{\{\alpha T_t x \mbox{ : }0\leq t\leq s_0\mbox{, }\alpha\in\mathbb{D}\}} $$
by using the relation 
$$ X=\overline{\{\alpha T_t x \mbox{ : }t\geq 0\mbox{, }\alpha\in\mathbb{D}\}}=\overline{\{\alpha T_t x \mbox{ : }t\geq s_0\mbox{, }\alpha\in\mathbb{D}\}}\cup \overline{\{\alpha T_t x \mbox{ : }0\leq t\leq s_0\mbox{, }\alpha\in\mathbb{D}\}}. $$
Thus, $\overline{U}$ is compact. Hence $X$ is finite dimensional, which contradicts that $X$ is infinite dimensional.
\end{proof}
\begin{theorem}\label{t6}
Let $(T_t)_{t\geq0}$ be a strongly continuous semigroup of operators on a complex separable Banach infinite dimensional space X. Then the following assertions are equivalent:
\begin{itemize}
\item[$(1)$]$(T_t)_{t\geq0}$ is diskcyclic;
\item[$(2)$]for all $y$, $z \in X$ and all $\varepsilon > 0$, there exist $v \in X$, $t>0$ and $\alpha \in\mathbb{D}$ such that
$$\Vert y-v \Vert<\varepsilon \hspace{0.3cm}\mbox{ and }\hspace{0.3cm} \Vert z-\alpha T_tv \Vert<\varepsilon;$$
\item[$(3)$]for all $y$, $z \in X$, all $\varepsilon > 0$ and for all $l\geq 0$,  there exist $v \in X$, $t>l$ and $\alpha \in\mathbb{D}$ such that
$$\Vert y-v \Vert<\varepsilon\hspace{0.3cm}\mbox{ and }\hspace{0.3cm}\Vert z-\alpha T_tv \Vert<\varepsilon.$$
\end{itemize}
\end{theorem}
\begin{proof}
$(1)\Rightarrow (3)$: Assume that $(T_t)_{t\geq0}$ is diskcyclic and let $x \in X$ be a diskcyclic vector for $(T_t)_{t\geq0}$. Then 
$$\mathbb{D}Orb((T_t)_{t\geq0},x)=\{ \alpha T_tx \mbox{ : }t\geq0 \mbox{, }\alpha\in\mathbb{D} \}$$
 is dense in $X$. Let $\varepsilon>0$. For any $y\in X$, there exist $s_1>0$ and $\alpha_1\in \mathbb{D}$ such that $\Vert y-\alpha_1 T_{s_1}x\Vert<\varepsilon $. If $l\geq0$, then by Lemma \ref{lm1}, the set  
$$\alpha_1 \{ \alpha T_tx \mbox{ : }t\geq s+l \mbox{, }\alpha\in\mathbb{D} \}:=\{\alpha_1 \alpha T_tx \mbox{ : }t\geq s+l \mbox{, }\alpha\in\mathbb{D} \}$$
 is a dense subset $X$. For any $z\in X$, there exist $s_2>l+s_1$ and $\alpha_2\in\mathbb{D}$ such that $\Vert z-\alpha_1\alpha_2 T_{s_2}x \Vert<\varepsilon$. Put $v=\alpha_1 T_{s_1}x$, $t=s_2-s_1>l$ and $\alpha=\alpha_2$. Then we have
$\Vert y-v \Vert<\varepsilon$ and $\Vert z-\alpha T_tv \Vert<\varepsilon.$\\
$(3)\Rightarrow (2)$: It is obvious.\\
$(2)\Rightarrow (1)$: Since $X$ is separable, we can consider $\{z_1,z_2,z_3,...\}$ a dense sequence in $X$. Using this sequence, we construct sequences $\{y_1,y_2,y_3,...\}$ of $X$, $\{t_1,t_2,t_3,...\}$ of $[0,+\infty)$ and $\{\alpha_1,\alpha_2,\alpha_3,...\}$ of $ \mathbb{D}$ inductively:
\begin{itemize}
\item Put $y_1=z_1$, $t_1=0$;
\item For $n>1$, find $y_n$, $t_n$ and $\alpha_n$ such that
\begin{equation}\label{e1}
\Vert y_n- y_{n-1} \Vert\leq \frac{2^{-n}}{\sup\{ \Vert T_{t_j} \Vert \mbox{ : } j<n \}},
\end{equation}
and 
\begin{equation}\label{e2}
\Vert z_n-\alpha_n T_{t_n}y_n \Vert\leq \varepsilon.
\end{equation}
\end{itemize}
In particular, (\ref{e1}) implies that $\Vert y_n- y_{n-1} \Vert\leq 2^{-n}$, so that the sequence $(y_n)_{n\geq1}$ has a limit $x$. Applying (\ref{e2}) and once again (\ref{e1}) we infer that
\begin{align*}
\Vert z_n-\alpha_nT_{t_n}x \Vert&=\Vert z_n-\alpha_nT_{t_n}y_n+\alpha_nT_{t_n}y_n-\alpha_nT_{t_n}x  \Vert\\
                        &\leq \Vert z_n-\alpha_nT_{t_n}y_n\Vert+\Vert \alpha_nT_{t_n}(y_n-x)  \Vert\\                                              
                        &\leq \Vert z_n-\alpha_nT_{t_n}y_n\Vert+\Vert \alpha_nT_{t_n}\Vert \left\| \sum_{i=n+1}^{+\infty} \Vert  y_i-y_{i-1}  \right\|\\
                           &\leq 2^{-n}+\sum_{i=n+1}^{+\infty}2^{-i}=2^{-n+1}.
\end{align*}
Let $z\in X$ and $\varepsilon>0$, there exists $n$, large enough, such that 
$$\Vert z_n-z \Vert<\frac{\varepsilon}{2}.$$
 Choosing $n$ large enough such that $2^{-n+1}<\frac{\varepsilon}{2}$, we obtain
$$\Vert \alpha_n T_{t_n}x-z \Vert\leq \Vert z-z_n \Vert+\Vert z_n -\alpha_n T_{t_n}x\Vert<\varepsilon.$$
Therefore, $\mathbb{D}Orb((T_t)_{t\geq0},x)=\{ \alpha T_tx \mbox{ : }t\geq0 \mbox{, }\alpha\in\mathbb{D} \}$ is dense in $X$. This means that is diskcyclic and $x$ is a diskcyclic vector for $(T_t)_{t\geq0}$.
\end{proof}
As a corollary we obtain a sufficient condition of diskcyclicity of a strongly continuous semigroup of operators.

Let $X$ be a separable Banach infinite dimensional space. Denote $ X_0$ the set of all $x\in X $ such that $\lim_{t\longrightarrow \infty }T_t x=0, $
and $ X_\infty$ the set of all $ x \in X$ such that for each  $\varepsilon > 0$ there exist some $w \in X,$  $\alpha\in\mathbb{D}$ and some $ t > 0$ with $ \Vert w \Vert < \varepsilon$ and $ \Vert\alpha T_t w -x \Vert < \varepsilon.$
\begin{theorem}
Let $(T_t)_{t\geq0}$ be a strongly continuous semigroup of operators on a complex separable Banach infinite dimensional space $X$. If both $X_\infty$ and $X_0$ are dense subsets, then $(T_t)_{t\geq0}$ is diskcyclic.
\end{theorem}
\begin{proof}
Let $z \in X_\infty$ and $y \in X_0$. Then for each $\varepsilon > 0$ there are
arbitrarily large $t > 0$, $\alpha\in\mathbb{D}$ and $w \in X$ such that
$$ \Vert w \Vert < \varepsilon \hspace*{0.3cm} \mbox{ and } \hspace*{0.3cm} \Vert\alpha T_t w -x \Vert < \frac{\varepsilon}{2}.$$
Since $y \in X_0$, for sufficiently large $t$ we have $\Vert T_{t}y \Vert < \frac{\varepsilon}{2}$. We put $v=y+w$ and infer
$$ \Vert z-\alpha T_t v \Vert\leq \Vert z-\alpha T_t w \Vert+\Vert\alpha T_t y \Vert<\varepsilon, $$
and 
$$ \Vert y-v \Vert=\Vert w \Vert<\varepsilon. $$
\end{proof}
By using Theorem \ref{so}, we may prove that diskcyclicity and disk transitivity are equivalent in the case of strongly continuous semigroup of operators.

\begin{theorem}\label{23}
Let $(T_t)_{t\geq0}$ be a strongly continuous semigroup of operators on a complex topological vector space $X$. Then, the following assertions are equivalent:
\begin{itemize}
\item[$(i)$] $(T_t)_{t\geq0}$ is diskcyclic set;
\item[$(ii)$] $(T_t)_{t\geq0}$ is disk transitive.
\end{itemize}
\end{theorem}
\begin{proof}
By remarking that if $t_1>t_2\geq0$, then there exists $t=t_1-t_2$ such that $T_{t_1}=T_t T_{t_2}$, and using Theorem \ref{so}.
\end{proof}
\begin{definition}\cite{Pazy}
Let $(T_t)_{t\geq0}$ be a strongly continuous semigroup of operators on a complex topological vector space $X$. Given another topological vector space $Y$ and an isomorphism $\phi$ from $Y$ onto $X$, we obtain a strongly continuous semigroup of operators $(S_t)_{t\geq0}$ on $Y$, called similar to $(T_t)_{t\geq0}$, by defining
$$ S_t=\phi^{-1} T_t \phi $$
for all $t\geq0$.
\end{definition}
\begin{proposition}
Let $(T_t)_{t\geq0}$ be a diskcyclic strongly continuous semigroup of operators on a complex topological vector space $X$. If $(S_t)_{t\geq0}$ is a strongly continuous semigroup of operators on $Y$ similar to $(T_t)_{t\geq0}$, then $(S_t)_{t\geq0}$ is diskcyclic on $Y$. Moreover, 
$$\mathbb{D}C((S_t)_{t\geq0})=\phi(\mathbb{D}C((T_t)_{t\geq0})).$$
\end{proposition}
\begin{proof}
Direct consequence of Proposition \ref{14}.
\end{proof}
\begin{definition}\cite{Pazy}
Let $(T_t)_{t\geq0}$ be a strongly continuous semigroup of operators on a complex topological vector space $X$. For any numbers $\mu\in\mathbb{C}$ and $\alpha > 0$, we define the rescaled strongly continuous semigroup of operators $(S_t)_{t\geq0 }$ by
$$ S_t= e^{\mu t}T_{(\alpha t)}$$
for all $t\geq0$.
\end{definition}
\begin{proposition}
Let $(T_t)_{t\geq0}$ be a diskcyclic strongly continuous semigroup of operators on a complex topological vector space $X$. For any reel number $\mu\geq0$, the rescaled strongly continuous semigroup of operators $(S_t)_{t\geq0 }=(e^{\mu t}T_{ t})_{t\geq0}$ is diskcyclic. 
\end{proposition}
\begin{proof}
Let $\mu$ be a positive reel number and $c_t=1\leq e^{\mu t}=k_t $ for all $t\geq0$. To conclude the result, we apply Proposition \ref{prop2} for $(c_t)_{t\geq0}$ and $(k_t)_{t\geq0}$.
\end{proof}

\section{Diskcyclic $C$-regularized groups of operators}

In this section, we study the particular case when $\Gamma$ stands for a $C$-regularized semigroup. 
Recall from \cite{CKM}, that an entire $C$-regularized group of operators is an operator family $(S(z))_{z\in\mathbb{C}}$ on $\mathcal{B}(X)$ that satisfies$:$
\begin{itemize}
\item[$(1)$] $S(0)=C;$
\item[$(2)$] $S(z+w)C = S(z)S(w)$ for every $z,$ $w\in\mathbb{C}$,
\item[$(3)$] The mapping $z \mapsto S(z)x$, with $z\in\mathbb{C}$, is entire for every $x \in X$.
\end{itemize}
An example of a diskcyclic $C$-regularized groups of operators on the complex field $\mathbb{C}$ is the following.
\begin{example}\label{example}
Let $X=\mathbb{C}$. For all $z\in\mathbb{C}$, let $S(z)x=\exp(z)x$, for all $x\in \mathbb{C}$. $(S(z))_{z\in\mathbb{C}}$ is a $C$-regularized group of operators and we have
$Orb((S(z))_{z\in\mathbb{C}},1)=\{\exp(z) \mbox{ : }z\in\mathbb{C}\},$ which implies that $(S(z))_{z\in\mathbb{C}}$ is diskcyclic and $1$ is a diskcyclic vector for $(S(z))_{z\in\mathbb{C}}$.
\end{example}
\begin{remark}
Since all complex topological vector spaces of dimension one are isomorphe, we can deduce, by Using Example \ref{example}, that there exists a diskcyclic $C$-regularized group of operators on each one dimensional space. 
\end{remark}
\begin{lemma}\label{lem}
Let 
 $(S(z))_{z\in\mathbb{C}}$ be a diskcyclic $C$-regularized group of operators in $X$.
 If $C$ is of dense range,
 then $Cx\in \mathbb{D}C((S(z))_{z\in\mathbb{C}})$, for all $x\in \mathbb{D}C((S(z))_{z\in\mathbb{C}}).$
\end{lemma}
\begin{proof}
Let $z\in \mathbb{C}$. By conditions $(1)$ and $(2)$ of Definition of  a $C$-regularized group of operators we have
$$ S(z)C=S(0+z)C=S(0)S(z)=CS(z),$$
this means that $C$ commutes with every element of $(S(z))_{z\in\mathbb{C}}$. Hence, $C\in \{(S(z))_{z\in\mathbb{C}}\}^{'}$. By using Proposition \ref{p1}, we deduce 
$Cx\in \mathbb{D}C((S(z))_{z\in\mathbb{C}})$, for all $x\in \mathbb{D}C((S(z))_{z\in\mathbb{C}}).$
\end{proof}
\begin{proposition}
Let $(S(z))_{z\in\mathbb{C}}$ be a diskcyclic $C$-regularized group on a complex Banach space.
If $C=I$ the identity operator on $X$, then $S(z)x\in \mathbb{D}C((S(z))_{z\in\mathbb{C}})$ for all $x\in Rec((S(z))_{z\in\mathbb{C}})$ and for all $z\in\mathbb{C}$.
\end{proposition}
\begin{proof}
By remarking that in this case we have $S(z)\in \{(S(z))_{z\in\mathbb{C}})\}^{'}$ and using Proposition \ref{p1}.
\end{proof}
\begin{definition}
Let $(S(z))_{z\in\mathbb{C}}$ be a $C$-regularized group on a complex topological vector space.
 Given another complex topological vector space $X$ and an isomorphism $\phi$ from $Y$ onto $X$, the $C$-regularized group $(h(z))_{z\in\mathbb{C}}$ on Y , defining by
 $$h(z)=\phi^{-1}S(z)\phi$$
 is said to be similar to $(S(z))_{z\in\mathbb{C}}$.
\end{definition}
\begin{proposition}
Let $(S(z))_{z\in\mathbb{C}}$ be a diskcyclic $C$-regularized group of operators on a complex topological vector space $X$.
 If $(h(z))_{z\in\mathbb{C}}$ is a  $C$-regularized group of operators on a complex topological vector space $Y$ similar to $(S(z))_{z\in\mathbb{C}}$, then
$(h(z))_{z\in\mathbb{C}}$ is diskcyclic on $Y$. Moreover,
$$\mathbb{D}C((S(z))_{z\in\mathbb{C}})=\phi(\mathbb{D}C((h(z))_{z\in\mathbb{C}})).$$
\end{proposition}
\begin{proof}
Direct consequence of Proposition \ref{14}.
\end{proof}

By Theorem \ref{t1}, every disk transitive $C$-regularized group is diskcyclic. In the following we prove the converse is also holds. 
\begin{theorem}
Let $X$ be a complex topological vector space and
 $(S(z)_{z\in\mathbb{C}})$ a $C$-regularized group such that $C$ has dense range. 
The following assertions are equivalent$:$
\begin{itemize}
\item[$(i)$] $(S(z)_{z\in\mathbb{C}})$  is diskcyclic; 
\item[$(ii)$] $(S(z)_{z\in\mathbb{C}})$ is disk transitive.
\end{itemize}
\end{theorem}
\begin{proof}
$(i)\Rightarrow(ii)$ : This implication is due to Theorem \ref{t1}.\\
$(i)\Rightarrow(ii)$ Assume that $(S(z))_{z\in\mathbb{C}}$ is diskcyclic, we will prove that $(S(z))_{z\in\mathbb{C}}$ is disk transitive.
Let $x\in \mathbb{D}C((S(z))_{z\in\mathbb{C}})$.
Let $U$ and $V$ be two nonempty open subsets of $X$, then there exist $\alpha$, $\beta$ $z_1$, $z_2\in\mathbb{C}$ such that
\begin{equation}\label{e11}
\alpha S(z_1)x\in C^{-1}(U) \hspace{0.6cm}\mbox{ and }\hspace{0.6cm}\beta S(z_2)x\in V. 
\end{equation}
Let $z_3=z_1-z_2$. By $\ref{e11}$, we have
$$\alpha S(z_3)(S(z_2)x)\in U\hspace{0.3cm}\mbox{ and }\hspace{0.3cm}\beta S(z_3)(S(z_2)x)\in  S(z_3)(V),$$
which implies that
$U\cap\frac{\beta}{\alpha} S(z_3)(V)\neq \emptyset.$
Hence, $(S(z))_{z\in\mathbb{C}}$ is a disk transitive $C$-regularized group.
\end{proof}

A necessary condition for a $C$-regularized group of operators to be diskcyclic is is due to next theorem. This result is similar to Lemma \ref{lm1} in the case of a strongly continuous semigroup of operators.
\begin{theorem}
Let $(S(z)_{z\in\mathbb{C}})$ be a diskcyclic $C$-regularized group on a Banach infinite dimensional space $X$. If $x \in X$ is a diskcyclic vector of $(S(z)_{z\in\mathbb{C}})$, then the following assertions hold:
\begin{itemize}\label{l1}
\item[$(1)$] $S(z) x\neq 0$ for all $z\in\mathbb{C}$;
\item[$(2)$] The set $\{\alpha S(z) x \mbox{ : }\alpha\in\mathbb{D}\mbox{, }z\in\mathbb{C}\mbox{; }\vert z\vert>\vert \omega_0\vert\}$ is dense in $X$ for all $\omega_0\in\mathbb{C}$.
\end{itemize}
\end{theorem}

\begin{proof}
$(1)$ If $z_1\in\mathbb{C}$ is such that $S(z_1)x=0$, then $S(z_1)(Cx)=0$. Let $z_2\in \mathbb{C}$. Then
$$ S(z_2)(Cx)=S(z_2-z_1+z_1)(Cx)=S(z_2-z_1)(S(z_1)x)=0, $$
which contradicts that $Cx\in \mathbb{D}C(S(z))_{z\in\mathbb{C}}).$\\
$(2)$ Suppose that there exists $\omega_0\in\mathbb{C}$ such that $A:=\{\alpha S(z)x \mbox{ : }\alpha\in\mathbb{D}\mbox{, }z\in\mathbb{C}\mbox{; }\vert z \vert>\vert \omega_0 \vert\}$ is not dense in X. Hence there exists a bounded open set $U$ such that $U \cap \overline{A}= \emptyset$.
Therefore we have
$$ U\subset \overline{\{\alpha S(z)x \mbox{ : }\alpha\in\mathbb{D}\mbox{, }z\in\mathbb{C}\mbox{; }\vert z \vert\leq\vert \omega_0 \vert\}} $$
by using the relation 
{\footnotesize
$$ X=\overline{\{\alpha S(z) x \mbox{ : }\alpha\in\mathbb{D}\mbox{, }z\in\mathbb{C}\}}=\overline{\{\alpha S(z)x \mbox{ : }\alpha\in\mathbb{D}\mbox{, }z\in\mathbb{C}\mbox{; }\vert z \vert>\vert \omega_0 \vert\}}\cup \overline{\{\alpha S(z)x \mbox{ : }\alpha\in\mathbb{D}\mbox{, }z\in\mathbb{C}\mbox{; }\vert z \vert\leq\vert \omega_0 \vert\}}. $$}
Since $S(z)x$ is continuous with $z$ and $S(z)x \neq0$  holds for all $z\in\mathbb{C}$ by $(1)$, there exist $m_1$, $m_2 >0$ such that $0 < m_1 \leq\Vert S(z)x\Vert < m_2$ for $z\in\mathbb{C}$ with $\vert z\vert\leq\vert \omega_0\vert$.
There exists $M > 0$ such that $\Vert y\Vert \leq M$ for any $y \in U$ because $U$ is bounded. So we have $$U \subset \overline{\left\lbrace \alpha S(z) x \mbox{ : }\vert z\vert\leq\vert \omega_0\vert \mbox{, }\vert \alpha \vert\leq \frac{M}{m_1} \right\rbrace  },$$
 which means that $\overline{U}$ is compact. Hence $X$ is finite dimensional, which contradicts the fact that $X$ is infinite dimensional.
\end{proof}


\begin{thebibliography}{00}
\bibitem {AOC} Amouch, M., Benchiheb, O.: On cyclic sets of operators, Rend. Circ. Mat. Palermo, 2. Ser (2018). https://doi.org/10.1007/s12215-018-0368-4.
\bibitem {AOH} Amouch, M., Benchiheb, O.: On linear dynamics of sets of operators, Turk. J. Math.,\textbf{43(1)}, 402-411 (2019)
\bibitem {AOS} Amouch, M., Benchiheb, O.: On supercyclic sets of operators, arXiv:1810.07577v1 [math.FA] 17 Oct 2018.
\bibitem {AKH}   Ansari, M., Hedayatian, K., Khani-robati, B.:
On the density and transitivity of sets of operators, Turk. J. Math., \textbf{42(1)}, 181-189 (2018)
\bibitem {AKH1} Ansari, M., Hedayatian, K., Khani Robati, B., Moradi, A.: A note on topological and strict transitivity, Iran. J. Sci. Technol. Trans. Sci., \textbf{42(1)},  59-64 (2018)
\bibitem {BKN} Bamerni, N., Kili\c{c}man, A., Noorani, M.S.M.: A review of some works in the theory of diskcyclic operators, Bull. Malays. Math. Sci. Soc., \textbf{39(2)}, 723-739 (2016) 
\bibitem{Bayart Matheron} Bayart, F., Matheron, E.: Dynamics of linear operators, New York, NY, USA, Cambridge University Press, 2009
\bibitem{CKM}  Conejero, J. A., Kostic, M., Miana, P. J., Murillo-Arcila, M.: Distributionally chaotic families of operators on Frechet spaces, Commun. Pure Appl. Anal., \textbf{15(5)}, 1915-1939 (2016)
\bibitem{Conway} Conway, J. B.: A Course in Functional Analysis, volume 96 of Graduate Texts in
Mathematics, Springer-Verlag, 1990
\bibitem{DSW} Desch, W., Schappacher, W., Webb, G. F.: Hypercyclic and chaotic semigroups of linear
operators, Ergod. Th. Dynam. Sys., \textbf{17}, 793-819 (1997)
\bibitem{Erdmann Peris} Grosse-Erdmann, K.-G.: Peris Manguillot, A.: Linear chaos, Universitext, Springer, London, 2011
\bibitem{HW} Hilden, H. M.: Wallen, L. J.:
 Some cyclic and non-cyclic vectors of certain operators, Indiana Univ. Math. J., \textbf{23}, 557-565 (1973/74)
\bibitem{LZ} Liang, Y. X., Zhou, Z. H.: Disk-cyclic and Codisk-cyclic tuples of the adjoint weighted composition operators on Hilbert spaces, Bull. Belg. Math. Soc. Simon Stevin., \textbf{23(2)}, 203-215. (2016)
\bibitem{LZ1}  Liang, Y. X., Zhou, Z. H.: Disk-cyclic and Codisk-cyclic of certain shift operators,
operators and matrices., \textbf{9(4)}, 831-846 (2015)
\bibitem{MT}  Matsui, M., Yamada, M., Takeo, F.: Erratum to Supercyclic and chaotic translation semigroups, Proc. Amer. Math. Soc. \textbf{132}, 3751-3752 (2004)
\bibitem{Pazy} Pazy, A.:
Semigroups of Linear Operators and Applications to Partial Differential Equations, Springer-Verlag, New York, 1983
\bibitem{Rolewicz} Rolewics, S.:
 On orbits of elements, Studia Math. \textbf{32}, 17-22 (1969)
\bibitem{WZ} Wang, Y., Zeng, H. G.: Disk-cyclic and codisk-cyclic weighted pseudo-shifts, Bull. Belg. Math. Soc. Simon Stevin. \textbf{25(2)}, 209-224 (2018)
\bibitem{Wengenroth} Wengenroth, J.:
Hypercyclic operators on non-locally convex spaces, Proc. Amer. Math. Soc. \textbf{131}, 1759-1761. (2003)



\end{thebibliography}
\end{document}